\documentclass[a4paper,12pt]{article}

\usepackage{latexsym}
\usepackage{amsmath,amsthm,amsthm,amsfonts}
\usepackage{graphicx}

\newcommand{\bN}{\mathbb{N}}

\newcommand{\cA}{\mathcal{A}}
\newcommand{\cC}{\mathcal{C}}
\newcommand{\cG}{\mathcal{G}}
\newcommand{\cT}{\mathcal{T}}
\newcommand{\cQ}{\mathcal{Q}}
\newcommand{\cP}{\mathcal{P}}
\newcommand{\cL}{\mathcal{L}}
\newcommand{\cI}{\mathcal{I}}

\newtheorem{theorem}{Theorem}[section]
\newtheorem{proposition}[theorem]{Proposition}
\newtheorem{lemma}[theorem]{Lemma}

\newtheorem{problem}[theorem]{Problem}

\theoremstyle{definition}
\newtheorem{example}[theorem]{Example}

\theoremstyle{remark}

\newtheorem*{Remark}{Remark}

\newenvironment{Proof}{\noindent {\bf Proof.} \ }{\hfill $\Box$ \bigskip}

\begin{document}

\title{Combinatorial configurations, quasiline arrangements,
      and systems of curves on surfaces}

\author{
  J\"urgen Bokowski,
  Jurij Kovi\v{c},   
  Toma\v{z} Pisanski, and  
  Arjana \v{Z}itnik        
}
\date{October 7,  2014}

\maketitle

\begin{abstract}
It is well known that not every combinatorial configuration admits a geometric
realization with points and lines. Moreover, some of them do not  even admit
realizations  with pseudoline arrangements, i.e., they are not topological.
In this paper we provide a new topological representation by using and 
essentially generalizing the topological representation 
of oriented matroids in rank 3. These representations can also be 
interpreted as curve arrangements on surfaces.

In particular, we generalize the notion of a pseudoline arrangement
to the notion of a quasiline arrangement by relaxing 
the condition that two pseudolines meet exactly once and show 
that every combinatorial configuration  can be realized as a 
quasiline arrangement in the real projective plane. 
We  also generalize well-known tools from pseudoline arrangements
such as sweeps or wiring diagrams.
A quasiline arrangement with selected vertices belonging to the configuration
can be viewed as a map on a closed surface. 
Such a map can be used to distinguish between two ``distinct" realizations 
of a combinatorial configuration as a quasiline arrangement.
\end{abstract}

\medskip
\noindent
{\sc Keywords}: pseudoline arrangement, quasiline arrangement, projective plane,
incidence structure, 
combinatorial configuration, topological configuration, geometric configuration,
sweep, wiring diagram, allowable sequence of permutations, maps on a surface.
\medskip

\noindent
{\sc Math. Subj. Class. (2010)}: 
51E20,   
52C30,   
05B30.   

\section{Introduction} 

For a long time researchers in configurations have neglected the difference 
between combinatorial and geometric configurations of points and lines. 
In the nineteenth century it was either tacitly assumed that these concepts 
coincide or the emphasis was on geometric existence of configurations
and the two well-known combinatorial configurations $(7_3)$ and $(8_3)$ 
were automatically excluded from consideration. 
On the other hand, graph-theoretical approach advocated by a series of papers 
by H. Gropp \cite{Gropp1,Gropp2} identified the concept of configuration
with a combinatorial configuration and the problem of existence of the corresponding 
geometric configuration reduced  to the problem of drawing the configuration.  

Two kinds of questions can be considered in this context:
\emph{is a given combinatorial configuration geometric?} and 
\emph{for which $n,k$ do there exist geometric configurations $(n_k)$?}
It was realized quite early by H. Schr\"oter \cite{Schroter2} that among the ten 
$(10_3)$ configurations, exactly one, that we call 
\emph{anti-Desargues configuration}, is not geometric. 
However, a geometric configuration $(v_3)$ exists if and only if 
$v \geq 9$; see for example  \cite[Theorem 2.1.3]{Grunbaum}.
B. Gr\"{u}nbaum also considered the problem of determining values of 
$v$ for which geometric $(v_4)$ configurations exist,
an old problem  first posed by T. Reye in 1882 \cite{Reye}. 
He proved existence for all but a finite number of values $v$ \cite{Grunbaum1}. 
J. Bokowski and L. Schewe \cite{BokowskiSchewe} resolved four previously unknown
cases, and J. Bokowski and V. Pilaud 
\cite{BokowskiPilaud,BokowskiPilaud1,BokowskiPilaud2} resolved  
another three open cases 
leaving only $v=22,23$, and $26$ open.

F. Levi \cite{Levi} introduced the notion of pseudolines and showed that the
combinatorial configurations $(7_3)$ and $(8_3)$ cannod be realized
with pseudoline arrangements  in the projective plane. 
B. Gr\"{u}nbaum was the first to observe a fundamental difference 
between the $(7_3)$ and $(8_3)$ configurations and the anti-Desargues 
configuration. Namely, the latter  can be realized as a pseudoline arrangement 
in the projective plane. The notion of \emph{topological configuration} was born. 
J. Bokowski and his coworkers were able to apply modern methods of
computational synthetic geometry to study the existence or nonexistence
of topological configurations \cite{BokowskiBook}. 
B. Gr\"{u}nbaum, J. Bokowski, and L. Schewe \cite{BokowskiGrunbaum} investigated 
the problem of existence of topological $(v_4)$ configurations and showed that 
topological configurations exist if and only if $v \geq 17$.
Recently J. Bokowski and R. Strausz \cite{BokowskiStrausz} associated to each 
topological configuration a map on a surface that they call a \emph{manifold 
associated to the topological configuration} and thus enabled the definition 
of equivalence of two topological configurations.
Namely, topological configurations are distinct in a well-defined sense 
if and only if the associated maps are distinct.

In this paper we first generalize the notion of a pseudoline arrangement 
to the notion of a \emph{quasiline arrangement} by relaxing 
the condition that two pseudolines meet exactly once, and  we define
a subclass of quasiline arrangements that we call \emph{monotone} (Section 4).
We  also generalize well-known tools from pseudoline arrangements
such as sweeps, wiring diagrams, and allowable sequences of permutations. 
It is known that every pseudoline arrangement can be represented by a wiring diagram and
conversely, every wiring diagram can be viewed as a pseudoline arrangement;
see  J.~E. Goodman \cite{handbook}. We show that  every monotone
quasiline arrangement can be represented by  a  generalized wiring diagram that
is in turn also a monotone quasiline arrangement (Sections 5 and 6).
In this respect the class of monotone quasiline arrangements is in some sense the weakest
generalization of the class of pseudoline arrangements.

In Section 7 we introduce a generalization of topological configurations that we call 
\emph{(monotone) quasi-topological configurations} by allowing the set of lines to form a 
(monotone) quasiline arrangement instead of a pseudoline arrangement.  
In Section 8 we show that every combinatorial incidence structure,
in particular every combinatorial configuration, can be
realized as  a monotone quasi-topological incidence structure. Moreover, we show that any
monotone quasi-topological configuration such that the underlying quasiline arrangement
has no digons, is topologically equivalent to a polygonal monotone quasi-topological 
configuration with no bends (arcs connecting two vertices of the arrangement are all straight lines).

Finally, the concept of a map associated with a topological configuration 
can be extended to the quasi-topological case (Section 9).  
Last but not least the theory we develop is applicable 
to general incidence structures and it is not limited to configurations.
\section{Basic definitions}
\label{section:definitions}

In this section we review basic definitions and facts
about incidence structures; see B. Gr\"unbaum \cite{Grunbaum}
or T. Pisanski and B. Servatius \cite{PisanskiServatius} for 
more background information.
An {\em incidence structure} $\cC$ is a triple $\cC = (\cP,\cL,\cI)$, where $\cP$
and $\cL$ are non-empty disjoint finite sets and  $\cI \subseteq \cP \times \cL$.
The elements of $\cP$  are called {\em points} and
the elements of $\cL$ are called  {\em lines}.
The relation $\cI$ is called {\em incidence relation};
if $(p,L) \in \cI$,  we say that the point $p$ is incident to the line $L$
or, in a geometrical language, that $p$ lies on $L$.
We further require that each point lies on at least two lines and
each line contains at least two points.
To stress the fact that these objects are of purely combinatorial
nature we sometimes call them \emph{abstract} or \emph{combinatorial}
incidence structures.

Two incidece structures  $\cC = (\cP,\cL,\cI)$ and 
${\cC}^\prime= ({\cP}^\prime,{\cL}^\prime,{\cI}^\prime)$ are \emph{isomorphic}, 
if there exists an incidence preserving bijective mapping 
from $\cP \cup \cL$ to ${\cP}^\prime  \cup {\cL}^\prime $
which maps  $\cP$ to ${\cP}^\prime$ and $\cL$ to ${\cL}^\prime$.


Complete information about the incidence structure can be recovered
also from its Levi graph with a given black and white coloring of the vertices.
The {\em Levi graph} $G(\cC)$ of an incidence structure $\cC$ is a bipartite graph
with ``black'' vertices $\cP$ and ``white'' vertices ${\cL}$ and with an edge
joining some $p \in \cP$ and some $L \in \cL$
if and only if  $p$ lies on $L$ in $\cC$.

An incidence structure is {\em lineal} if any two distinct
points are incident with at most one common line.
This is equivalent to saying that the Levi graph of a lineal incidence
structure has girth at least $6$.
A {\em$(v_r,b_k)$-configuration} is a lineal incidence structure
$\cC = (\cP, \cL, \cI)$ with $\vert \cP\vert  = v$ and   $\vert \cL\vert = b$
such that each line is incident with the same number $k$
of points and each point is incident with the same number $r$ of lines.
In the special case when $v = b$ (and by a simple counting argument also
$r = k$) we speak of a {\em balanced} configuration and shorten
the notation $(v_k,v_k)$ to $(v_k)$.

A set of lines in the real Euclidean or projective plane 
together with a subset of intersection points of these lines
such that  each line contains at least two intersection points
is called a \emph{geometric incidence structure}.
From the definition it follows that each point lies on at least two lines.
A geometric incidence structure together with the incidences of points
and lines  defines a combinatorial incidence structure,
which we call the \emph{underlying combinatorial incidence structure}.
Such a  combinatorial incidence structure is certainly lineal.
Two geometric incidence structures are \emph{isomorphic}
if their underlying combinatorial incidence structures are isomorphic.
A geometric incidence structure  $\cG$ is a \emph{realization} of
a combinatorial incidence structure $\cC$ if the underlying
combinatorial incidence structure of $\cG$ is isomorphic to $\cC$.

\section{Pseudolines and topological incidence structures}
\label{sec:topological}

In this section we review  basic facts about pseudoline arrangements.
By a projective plane we mean the real projective plane or the
extended Euclidean plane.
A \emph{pseudoline} is a simple non-contractible closed curve
in the projective plane. In particular, each line in the projective plane
is a pseudoline. 

Pseudolines and certain relationships between them inherit
properties from the topological structure of the projective plane \cite{handbook}. 
For instance:\\

\noindent
\textbf{Fact I.} Any two pseudolines have at least one point in common. \\

\noindent
\textbf{Fact II.} If two pseudolines meet in exactly one point
they intersect transversally at that point. \\

A \emph{pseudoline arrangement} $\cA$ is a collection of at least two pseudolines
in the projective plane with the property that each pair of pseudolines of $\cA$
has exactly  one point in common (at which they cross transversally).
Such a point of intersection is called a \emph{vertex} or a 
\emph{crossing} of the arrangement.
A crossing in which only two pseudolines meet is called \emph{regular}.
If more than two pseudolines meet in the same point, the crossing is called
\emph{singular}. 
Each pseudoline arrangement $\cA$ determines an associated   2-dimensional
cell complex into which the pseudolines of  $\cA$ decompose the projective plane.
Its cells of dimension 0,1,2 are called \emph{vertices}, \emph{edges},
and \emph{cells} (or \emph{polygons}), respectively; 
see Gr\"unbaum \cite[p. 40]{Grunbaum72}.
Two pseudoline arrangements are \emph{isomorphic} if the associated cell
complexes are isomorphic; that is, if and only if there exists an
incidence preserving bijective mapping between the vertices, edges, and
cells of one arrangement and those of the other.

We say that a pseudoline is \emph{polygonal}, if it is a line or
it can be subdivided into a finite number of closed line segments
(whereby the endpoints of the line segments occur of course twice).
A pseudoline arrangement is \emph{polygonal} if every pseudoline 
of the arrangement is polygonal. Note that a line arrangement is polygonal. 
A point on a polygonal pseudoline that is not a crossing of the arrangement
is called a \emph{bend} if two line segments meet at that point
and the join of these segments is not again a line segment.
A crossing $v$ of a polygonal pseudoline arrangement is \emph{straight}, 
if there exists a neighborhood $N(v)$ of $v$ such that the intersection of $N(v)$
with every   pseudoline $\ell$ containing $v$ is a line segment.


To describe pseudoline arrangements combinatorially,
wiring diagrams   are standard tools to use; see J.~E. Goodman \cite{handbook}
and J.~E. Goodman and R. Pollack \cite{GoodmanPollack}. 

\emph{A partial wiring diagram} is a collection of $x$-monotone polygonal lines
in the  Euclidean plane, each of them horizontal except for a finite number of short
segments, where it crosses another polygonal line. 
\emph{A wiring diagram} is a partial wiring diagram 
with the property  that every two polygonal lines cross exactly once.

A wiring diagram can be viewed as a pseudoline arrangement: take a disk that
is large enough that all the crossings are in its interior, and positioned
such that the intersections of each polygonal line with the boundary of the disk
are on opposite sides of the boundary of the disk. The disk can now be
viewed as a disk model of the projective plane and the lines of
the wiring diagram form  a  polygonal pseudoline arrangement.

Conversely,  every pseudoline arrangement
can be described with a  wiring diagram; see \cite{handbook}.

\begin{theorem}                              \label{cor:pseudowiring}
Every pseudoline arrangement is isomorphic 
to a  wiring diagram.
\end{theorem}

A pseudoline arrangement can also be viewed as an incidence structure 
when we define   a  subset of its crossings as its point set.
We say that an incidence structure is \emph{topological} 
if its points are points in the projective plane, 
and lines are pseudolines that form a pseudoline arrangement.
A topological incidence structure  $\cT$ is a \emph{topological realization} of
a combinatorial incidence structure $\cC$ if the underlying
combinatorial incidence structure of $\cT$ is isomorphic to $\cC$.
A topological incidence structure is \emph{polygonal} if its  lines are 
polygonal pseudolines.
Note that any geometric incidence structure is also (polygonal) topological.

There are three distinct notions of equivalence of topological incidence structures.
The weakest is combinatorial equivalence.
Two topological incidence structures  are \emph{combinatorially equivalent}
or \emph{isomorphic} if they are isomorphic as combinatorial incidence structures.
The strongest one is the  notion of topological equivalence between 
pseudoline arrangements in the projective plane. Two topological 
incidence structures are \emph{topologically equivalent} if 
there exists an isomorphism of the underlying pseudoline arrangements 
that induces an isomorphism of the underlying combinatorial incidence structures.

One intermediate notion is mutation equivalence. 
A \emph{mutation} or a \emph{Reidemeister move} in a pseudoline arrangement 
is a local transformation of the arrangement where only one pseudoline $\ell$ 
moves across a single crossing $v$ of the remaining arrangement.
Only the position of the crossings of $\ell$ with the pseudolines
incident to $v$ is changed. If those crossings are not points of the incidence structure,
we say that such a mutation is \emph{admissible}.
Two topological incidence structures are \emph{mutation equivalent}
if they can be modified by (possibly empty) sequences of admissible mutations 
to obtain topologically equivalent topological incidence structures.

It is well-known that every topological configuration can be realized 
as a polygonal pseudoline arangement \cite{Grunbaum}.
In \cite[Theorem 3.3]{Grunbaum72} essentialy the following theorem
was presented and a proof by induction was proposed.

\begin{proposition}               \label{prop:nobends}
Any  topological incidence structure is topologically equivalent
to a polygonal topological incidence structure with no bends.
\end{proposition}

\section{Quasiline arrangements}
\label{section:quasiline}

In this section we generalize the notion of a pseudoline arrangement
in which we relax the condition on pseudoline crossings.
A \emph{quasiline arrangement} is a collection of at least two pseudolines
in the real projective plane with the property that any two pseudolines
have a finite number of points in common and that at each common point
they cross transversally. Note that any pair of pseudolines in a 
quasiline arrangement meets an odd number of times.
The terms such as crossings, bends, polygonal quasiline arrangements,
and isomorphic quasiline arrangements are defined in the same way as 
for  pseudoline arrangements.

To simplify the discussion we will consider only quasiline arrangements 
in the extended Euclidean plane 
with the following additional properties:
\begin{itemize}
\item none of the crossings  of the arrangement are points at infinity and
\item every pseudoline of the arrangement intersects the line at
      infinity exactly once.
\end{itemize}

The former condition can  easily be achieved in general by a suitable
projective transformation while the latter is an essential assumption.
Namely,  a pseudoline can intersect the line at infinity more than
once. 


In the sequel we define a subclass of  quasiline arrangements,
called monotone quasiline arrangements, that is in
some sense the least generalization of the pseudoline arrangements.

It is well-known that the extended Euclidean plane is in its topological
sense equivalent to the disk model of the projective plane:
the projective plane is represented by a disk, 
with all the pairs of antipodal points on the boundary identified. 
The boundary of the disk is the line at infinity
of the projective plane and  the points on the line at infinity
are points at infinity.
For the rest of the section we will use the disk model of the projective plane.

%

%

Let  $\cA$ be  a  quasiline arrangement.
We choose  an orientation for the line at infinity $\ell$ and 
a point $x$ on  $\ell$ that is not on any of the pseudolines of the arrangement. 
Denote by $x^+$, $x^-$ the corresponding points on the boundary of the disk 
representing the projective plane without identifying $x^+$ and $x^-$. 
These two points divide the boundary of the disk into two arcs, 
$\ell^+$ from $x^+$ to $x^-$ and $\ell^-$ from $x^-$ to $x^+$,
again by forgetting the identification of antipodal points.
Starting at $x^+$ and moving along $\ell$ in the positive direction 
we orient a pseudoline that we meet for the first time 
such that it points to the interior of the circle. We call such an
orientation  a \emph{monotone orientation} of the quasiline arrangement $\cA$
and call the arrangement together with the point $x^+$
(or equivalently, with  a monotone orientation) a \emph{marked arrangement}. 
In a marked arrangement we have a natural ordering of the pseudolines
as the order of intersections  with $\ell$ starting at $x^+$.
Moreover, the orientation of pseudolines induces the order of crossings on every pseudoline.
A marked arrangement $({\cA},x^+)$ is \emph{proper} if the order of  
the intersection points of any two pseudolines is the same for both pseudolines.

A quasiline arrangement $\cA$ is a \emph{monotone quasiline arrangement} if
there exists a proper marked arrangement $({\cA},x^+)$ for some $x^+$ on the
boundary of the disk, representing the line at infinity.
Note that it is not the same to require that there exists an orientation 
of the pseudolines such that the order of  the intersection points 
on any  two pseudolines is the same for both pseudolines. Figure \ref{fig:notquasi}
shows a quasiline arrangement in which the order of crossings 
is the same for any pair of pseudolines if the pseudolines are oriented
alternatingly, however, it is not a monotone quasiline arrangement.
The orientation of the horizontal line induces the orientations on the other pseudolines.
This implies that there is no position for point $x^+$ on the line at infinity for a proper  
marked arrangement.

\begin{figure}[htb]
\centering
  \includegraphics[width=0.35\textwidth]{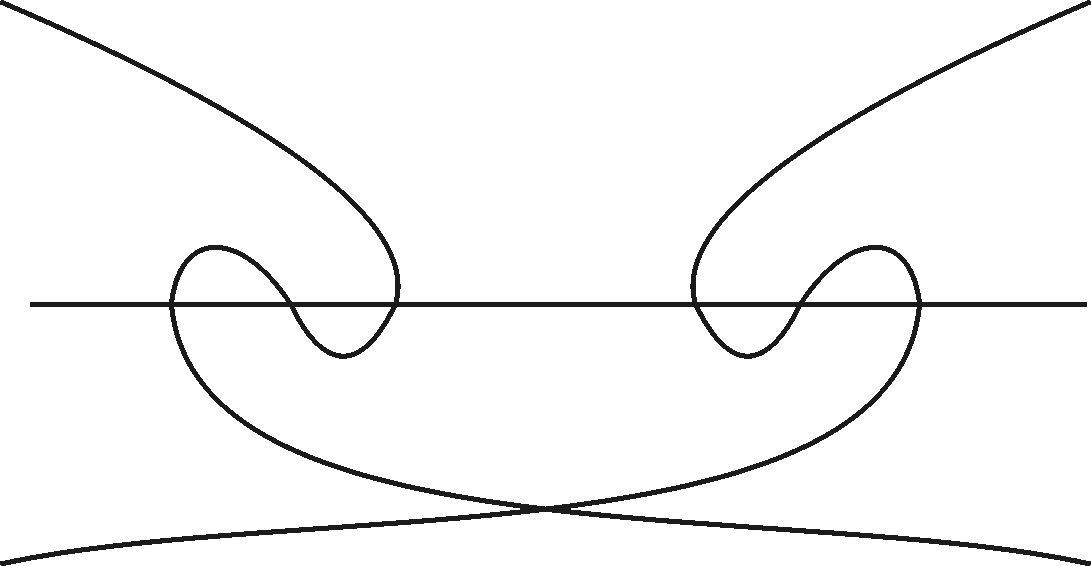}
\caption{A  quasiline arrangement that is not a monotone quasiline arrangement.}
\label{fig:notquasi}
\end{figure}

We now generalize the notion of the  wiring diagram from Section \ref{sec:topological}
to be able to describe also monotone quasiline arrangements.
\emph{A generalized wiring diagram} is a partial wiring diagram 
with the property that every two polygonal lines cross an odd number of times.
A generalized wiring diagram can be viewed as a monotone polygonal quasiline arrangement
just like a wiring diagram can be viewed as a pseudoline arrangement. 
Figure  \ref{wiring} shows a generalized wiring diagram which is topologically
equivalent to the quasiline arrangement from Figure ~\ref{fig:fano} (a). 
Observe that by adding seven points (corresponding to the crossings where three 
pseudolines cross) we arrive at the $(7_3)$ configuration.

\begin{figure}[htb] 
\centering {\includegraphics[trim = 0mm 0mm 0mm 5mm, clip, scale=0.7]{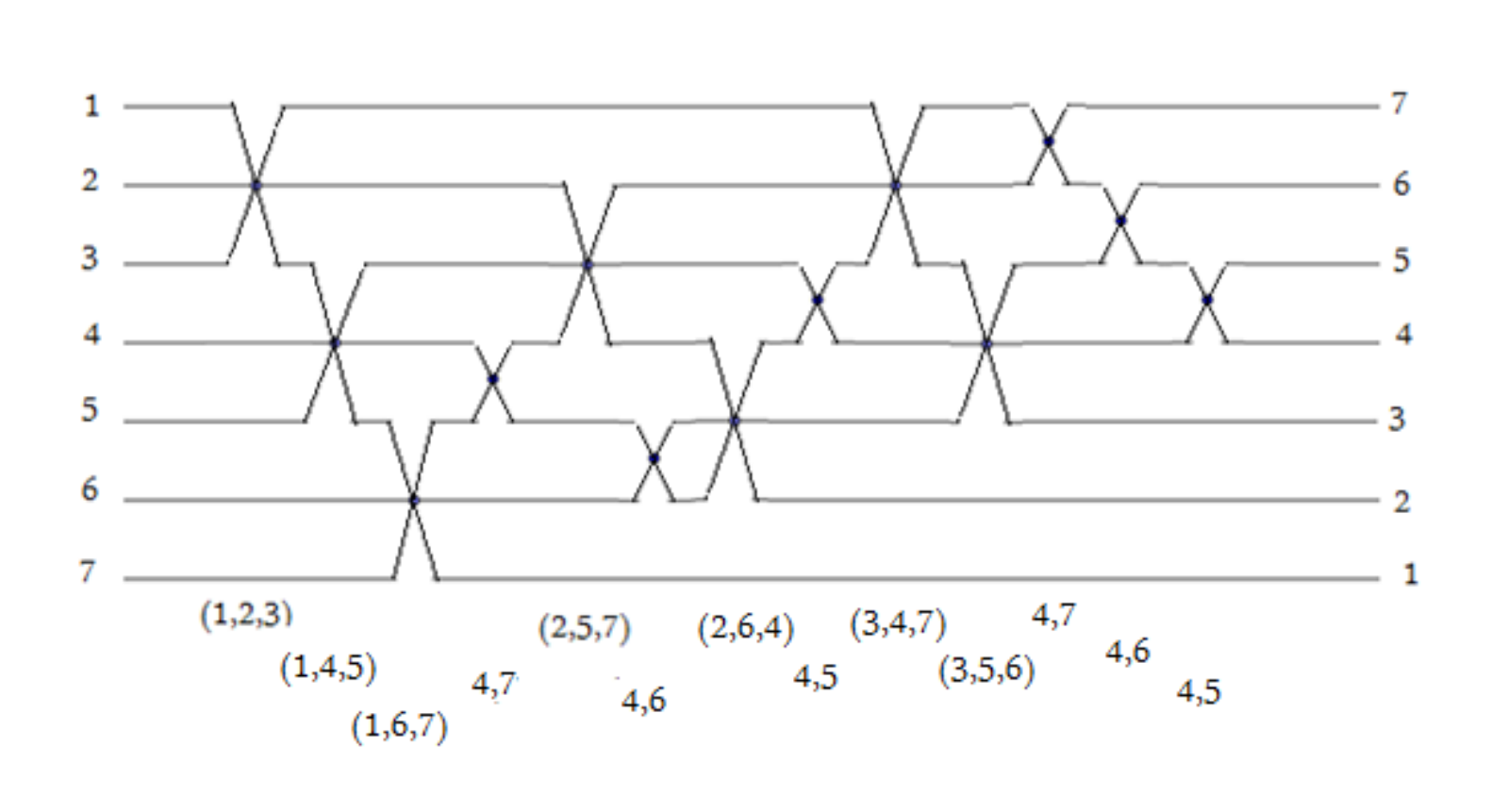}}
\vspace*{-0.5cm}
 \caption{A generalized wiring diagram of the Fano plane from Figure ~\ref{fig:fano} (a)}
 \label{wiring}
\end{figure}

One of the main results of the paper is the following generalization of
Theorem    \ref{cor:pseudowiring}.

\begin{theorem}                               \label{cor:monotonewiring}
Every monotone quasiline arrangement is
isomorphic  to a generalized wiring diagram.
\end{theorem}

We will prove the theorem in Section \ref{section:wiring}.
To do this we will need the notions of sweeping
and allowable sequences of permutations, which we introduce
in the next two sections.

\section{Sweeping quasiline arrangements}
\label{section:sweeping}

In this section we  introduce the notion of sweeping and show that every
monotone quasiline arrangement has a sweep.
With a slight modification working in the disk model of the projective
plane instead of the Euclidean plane, we follow S. Felsner and H. Weil 
\cite{Felsner}.

Let $({\cA}, x^+)$ be a proper marked quasiline arrangement.
A \emph{sweep} of $({\cA}, x^+)$ is a sequence 
$c_0, c_1, \dots, c_r$ , of pseudolines such that the following conditions hold:
\begin{itemize}
\item[(1)] each pseudoline $c_i$ crosses the line at infinity exactly once at $x$,
\item[(2)] none of the pseudolines $c_i$ contains a vertex of the arrangement $\cA$,
\item[(3)] each pseudoline $c_i$ has exactly one point of intersection 
           with each pseudoline of $\cA$,
\item[(4)] any two pseudolines $c_i$ and $c_j$ intersect exactly once at $x$,
\item[(5)] for any two consecutive pseudolines $c_i, c_{i+1}$ of the sequence
           there is exactly one vertex of the arrangement $\cA$ between them,
           i.e.,  in the interior of the region bounded by $c_i$ and $c_{i+1}$,
\item[(6)] every vertex of the arrangement is between a unique pair of consecutive 
           pseudolines, so the interior of the region bounded by $c_0$ and $c_{r}$ 
           contains all the vertices of  $\cA$.
\end{itemize}

We will  show that every monotone quasiline arrangment has a sweep.
To this end we define  a directed graph, or briefly a digraph,  
$D = D({\cA}, x^+)$ that corresponds to a marked quasiline arrangement 
$({\cA}, x^+)$ as follows. 
The vertices of $D$ are the vertices of $\cA$ and there is a directed edge 
for every pair of vertices  $u$ and $v$ that are consecutive on an arc of some
pseudoline oriented from $u$ to $v$ that has an empty intersection
with the line at infinity. We call such a digraph to be \emph{associated with}
the marked arrangement $\cA$. Note that this digraph is embedded in the plane.
In \cite{Felsner} S. Felsner and H. Weil prove that the digraph associated with 
a marked arrangement of pseudolines is acyclic.
We prove the following generalization of this result.

\begin{lemma}                                \label{lem:acyclic}
The digraph associated with a proper marked quasiline arrangement is acyclic.
\end{lemma}
\begin{Proof}
Without loss of generality we may  assume that the line at infinity is oriented counterclockwise.
Let $({\cA}, x^+)$ be a proper marked quasiline arrangement and
$D$ its associated digraph. Note that $D$ is a plane graph, embedded in
the interior of the disk, representing the projective plane.
The interior of the disk is homeomorphic to the plane,
therefore the Jordan curve theorem applies. 

First we observe that $D$ contains no directed cycles of length 2, 
otherwise there are two lines with different orders of vertices on them 
and $({\cA}, x^+)$ is not a proper marked arrangement. For the same reason
there are no directed cycles, all the edges of which belong to two pseudolines.

Suppose $D$ contains a directed cycle. Let $C$ be a  directed cycle, given by the
sequence of vertices and edges $v_0,e_0,v_1,\dots,e_{t-1},v_t=v_0$ such that
no other directed cycle is contained in the area bounded by $C$.
It is easy to see that $C$ bounds a face of $D$. 
Since at each vertex of the arrangement 
there meet at least two pseudolines, two consecutive edges of $C$ lie
on different pseudolines.  
Suppose that $C$ is oriented clockwise; see Figure \ref{Fig:acyclic}.
If $C$ is oriented counterclockwise, the proof is similar.

Now consider the arrangement $\cA^\prime$ consisting only of the pseudolines 
$\ell_0,\dots,\ell_{t-1}$ containing the edges $e_0,\dots,e_{t-1}$ of $C$ consecutively.
Note that also  $(\cA^\prime,x^+)$ is a proper marked arrangement.
Denote by $p_i$  the intersection of $\ell_i$ with the line at infinity $\ell$ for each $i$. 
We observe that $\ell_i$ and $\ell_{i+1}$ are distinct 
(the indices are taken modulo $t$), since two consecutive edges of $C$ lie on
different pseudolines.
Without loss of generality we may assume that $p_0^+$ appears first  after $x^+$.

Since $p_1^+$ appears after $p_0^+$ on the line at infinity, the 
arc of $\ell_1$ from  $v_1$ to $p_1^+$ must intersect $\ell_0$ to reach $\ell^+$.
It can intersect $\ell_0$ only between $p_0^+$ and $v_0$ (an odd number of times), 
otherwise we obtain a directed cycle contained in only two pseudolines.
Similarly, since $p_1^-$ appears after $p_0^-$ on the line at infinity, the 
arc of $\ell_1$ from $v_2$ to $p_1^-$ must intersect $\ell_0$ to reach $\ell^-$.
It can intersect $\ell_0$ only between $v_1$ and $p_0^-$ (an odd number of times); 
see Figure \ref{Fig:acyclic}.

\begin{figure}[htb] 
\centering {\includegraphics[scale=0.3]{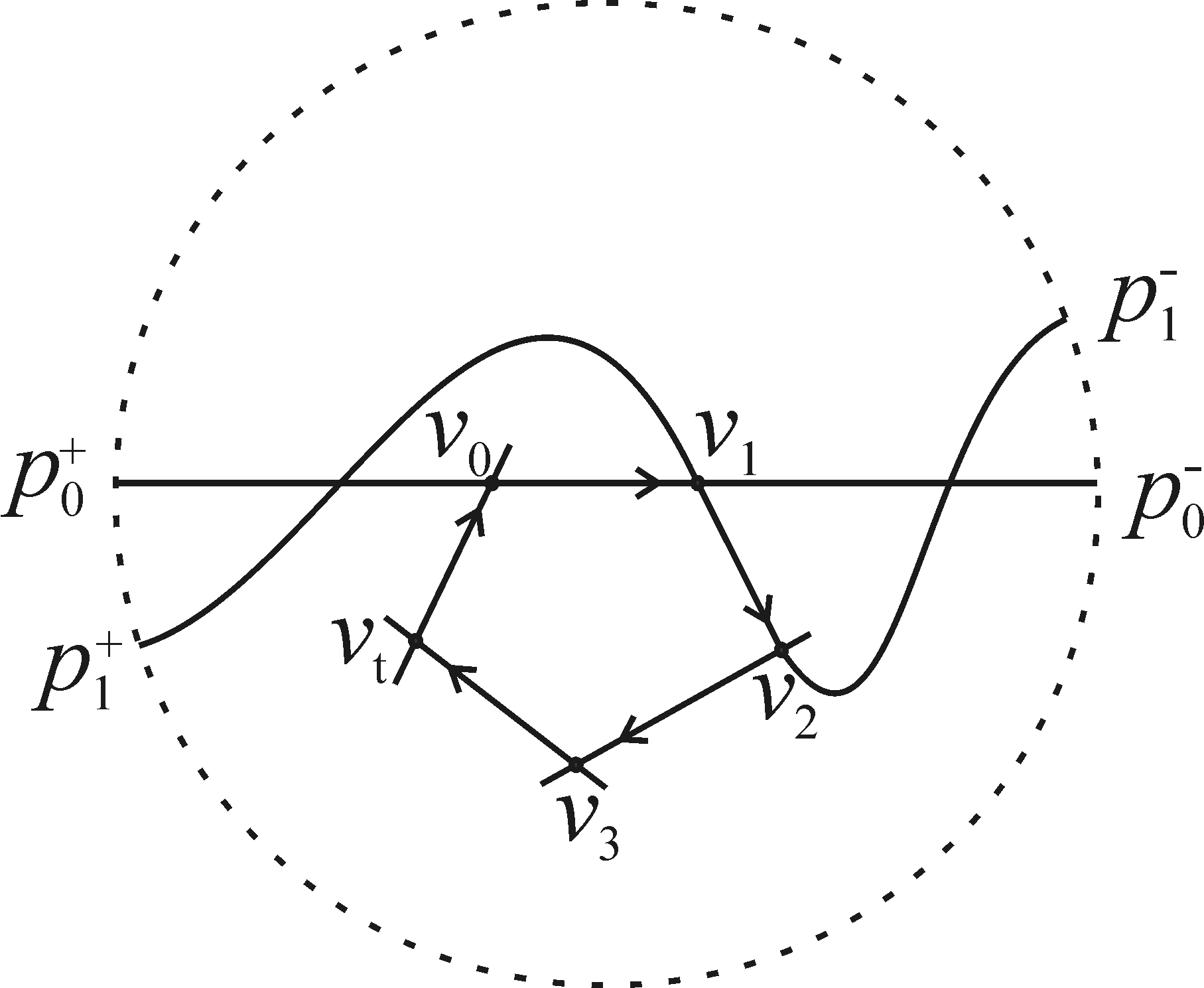}}
\vspace*{-0.2cm}
\caption{A  directed cycle in a marked quasiline arrangement.}
\label{Fig:acyclic}
\end{figure}

Now the arc of $\ell_2$ between $v_2$ and $p_2^+$ cannot intersect $\ell_1$
after $v_2$, since then we again obtain a directed cycle on two lines.
Therefore it must intersect $\ell_0$ after $v_1$ and then again before $v_0$
to reach $\ell^+$. If $t>2$, we see that also the arc of $\ell_2$
between $v_3$ and $p_2^-$ must intersect $\ell_0$ to reach $\ell^-$.
In order to avoid self-intersection it can cross $\ell_0$ only after $v_1$ (at least once).
With the same reasoning we conclude that for each of $\ell_i$, $i=2,\dots,t-2$,
\begin{itemize}
\item the arc of $\ell_i$ between $v_{i}$  and  $p_i^+$ first intersects  
      $\ell_0$ after $v_1$ exactly once and then again before $v_0$ at least once;
\item the arc  of $\ell_i$ between $v_{i+1}$ and $p_i^-$  intersects 
       $\ell_0$ only after $v_1$  (at least once).
\end{itemize}

For the line $\ell_{t-1}$ there are two cases to consider.
\begin{itemize}
\item The pseudolines $\ell_{t-1}$ and $\ell_1$ are distinct.
      As before, the arc of $\ell_{t-1}$ between $v_{t-1}$ and $p_{t-1}^+$ 
      first intersects  $\ell_0$ after $v_1$ exactly once and then again 
      before $v_0$ at least once. But then the arc of $\ell_{t-1}$ after $v_0$ 
      cannot reach $\ell^-$ without intersecting $\ell_0$ before $v_0$ 
      thus producing a directed cycle on two pseudolines, or self-intersection.
      A contradiction.
\item The pseudoline $\ell_{t-1}$ is the same as $\ell_1$. 
      In that case the vertices $v_{t-2},v_0$ and $v_1$ follow
      in that order on $\ell_1$, otherwise we don't have a proper marked arrangement.
      But then the arc of $\ell_{t-2}$ from $v_{t-2}$ to $p_{t-2}^-$ must
      intersect $\ell_1$ before $v_{t-2}$, thus producing a directed cycle
      on two pseudolines. A contradiction.
      \end{itemize}
\end{Proof}

\begin{lemma}                         \label{thm:sweep}
Every monotone quasiline arrangement has a sweep.
\end{lemma}
\begin{Proof}
Let  $\cA$ be a monotone quasiline arrangement and let $({\cA}, x^+)$ 
be a proper marked quasiline arrangement corresponding to $\cA$.
By Lemma \ref{lem:acyclic}, its associated digraph $D$ is acyclic.
Therefore there exist a topological sorting $v_1, \dots, v_r $
of the vertices  of $D$. 

We will define a sweep consisting of pseudolines $c_0, c_1, \dots, c_r$
starting at $x^+$ and ending at $x^-$, such that in the interior 
of the region bounded by $c_{i-1}$ and $c_{i}$, $i=1,\dots,r$, 
there will be exactly one vertex of the arrangement, namely $v_i$. 

Define $c_0$ to be the pseudoline that starts at $x^+$ and ends at $x^-$at 
and is the right boundary of an $\epsilon$-tube around the line at infinity.

Suppose that $c_{i-1}$ has been defined for some $i \le r$.
Let $\ell_1, \dots, \ell_t$ be the pseudo\-lines of the arrangement $\cA$
that contain $v_i$, in the order they intersect $c_{i-1}$.
Take the triangle $T$ with sides on $c_{i-1}$, $\ell_1$ and $\ell_t$
and one of the vertices being $v_i$. This is well defined, 
since $\ell_1$ and $\ell_t$ intersect $c_{i-1}$ only once 
by definition of $c_{i-1}$. Only vertices $v_1,\dots, v_{i-1}$ of the
arrangement (and all of them) are on the other side of $c_{i-1}$ as $v_i$, 
therefore all the lines $\ell_1, \dots, \ell_t$ are directed towards $v_i$
and there are no other vertices of the arrangement in the triangle $T$
besides $v_i$. Define $c_i$ to be the right boundary of an $\epsilon$-tube
around $c_{i-1}$ and $T$.  If $\epsilon$ is small enough, only vertex $v_i$ 
will be in the interior of the region bounded by $c_{i-1}$ and $c_{i}$
and $c_{i}$ will intersect every line of the arrangement only once.

Clearly the pseudolines $c_0, c_1, \dots, c_r$ obtained in this way
define a sweep of the marked arrangement $({\cA}, x^+)$.
\end{Proof}

\section{Generalized allowable sequences and wiring diagrams}
\label{section:wiring}

Wiring diagrams  and allowable sequences are standard tools
for describing pseudoline arrangements; see J.~E. Goodman \cite{handbook}
and J.~E. Goodman and R. Pollack \cite{GoodmanPollack}. 
We need to generalize these two notions in order to be able to describe
also  monotone quasiline arrangements and monotone quasi-topological
incidence structures.
We introduced generalized wiring diagrams in Section \ref{section:quasiline}.
In this section we generalize also the notion  of an allowable sequence
and show how generalized wiring diagrams and generalized allowable sequences
are related.

Fix $n \in \bN$.
A sequence $\Sigma=\pi_0, \dots, \pi_r$ of permutations
is called a  \emph{partial allowable sequence of permutations}
if it fulfills the following  properties:
\begin{itemize}
\item[(1)] $\pi_0$ is the identity permutation  on $\{1, \dots, n\}$,
\item[(2)] each permutation $\pi_i$, $1 \le i \le r$, is obtained
           by the reversal of a consecutive substring $M_i$
           from   the preceding permutation $\pi_{i-1}$.
\end{itemize}
For each $i$, $1 \le i \le r$, we call the transition
from the permutation $\pi_{i-1}$ to $\pi_i$ a \emph{move}. 
We will  also call each  $M_i$ a \emph{move}. 

Two partial allowable sequences $\Sigma$ and $\Sigma^\prime$ are
\emph{elementary equivalent} if $\Sigma$ can be transformed into $\Sigma^\prime$
by interchanging  two disjoint adjacent moves.
Two partial allowable sequences $\Sigma$ and $\Sigma^\prime$ are
called \emph{equivalent} if there exists a sequence
$\Sigma=\Sigma_1, \Sigma_2, \dots, \Sigma_m = \Sigma^\prime$ 
of partial allowable sequences such that
$\Sigma_i$ and $\Sigma_{i+1}$ are elementary equivalent for $1 \le i < m$.


A partial allowable sequence of permutations
$\Sigma=\pi_0, \dots, \pi_r$ is called an  \emph{allowable
sequence of permutations} if any two elements $x,y \in \{1,\dots n \}$ 
are joint members of exactly  one move $M_i$.

\begin{proposition}
Let $\Sigma=\pi_0, \dots, \pi_r$ be an allowable
sequence of permutations. Then $\pi_r$ is the reverse permutation
on $\{1, \dots, n\}$.
\end{proposition}

A partial  allowable sequence of permutations
$\Sigma=\pi_0, \dots, \pi_r$ is called a  \emph{qeneralized allowable
sequence of permutations} if $\pi_r$ is the reverse permutation
on $\{1, \dots, n\}$.

\begin{proposition}
A partial allowable sequence of permutations
$\Sigma=\pi_0, \dots, \pi_r$ is a generalized allowable
sequence of permutations if and only if any two elements
$x,y \in \{1,\dots n \}$  are joint members of an odd number of  moves $M_i$.
\end{proposition}


To any partial allowable sequence there corresponds a partial wiring diagram:
\begin{itemize}
\item start drawing  $n$ horizontal lines, numbered with numbers $1,\dots, n$ 
      from top to bottom, at points $(0,n),\dots,(0,1)$,
\item to each move  $M_i$ there corresponds coordinate  $x_i=i$,
\item at each coordinate $x_i$  there is a crossing where the lines in the 
      move  $M_i$ cross transversally; i.e., the order of lines from $M_i$ is reversed,
\item we extend each of the polygonal lines to infinity at both sides horizontally
        (where they meet in a common point).
\end{itemize}

Conversely, to any partial wiring diagram there corresponds a partial
allowable sequence in the following way. We number the lines of the
wiring diagram from top to bottom with numbers $1,\dots, n$.
We start with the identity permutation and after each crossing of the wiring diagram 
we list the lines from top to bottom. 

Obviously, a partial wiring diagram, corresponding to an allowable sequence is
also a wiring diagram. A partial wiring diagram, corresponding to a 
a generalized allowable sequence is also a generalized wiring diagram.

Now we can prove Theorem \ref{cor:monotonewiring}.\\

\noindent
\emph{Proof of Theorem} \ref{cor:monotonewiring}.
Let $\cA$ be a monotone quasiline arrangement  with $n$ pseudolines.
Then there exists a proper marked arrangement $({\cA}, x^+)$ for some
$x$ on the line at infinity.
We may label the lines of ${\cA}$ by numbers $1,\dots,n$ 
in the order in which they are met on the line at infinity starting at $x^+$. 
Take a sweep $c_0, c_1, \dots, c_r$ of $({\cA}, x^+)$. 
It determines a sequence of permutations on the set $\{1,\dots,n\}$:
just list the lines in the order in which they are met on each 
$c_i$, $i=0,\dots,r$, starting at $x^+$.
Since between any two pseudolines $c_i$ and $c_{i+1}$ there is exactly one
vertex of the arrangement, and the order of lines that meet in that vertex
is reversed when they leave the vertex, this sequence of permutations 
is a partial allowable sequence of permutations.
Moreover, since the order of the lines of ${\cA}$ is reversed at $c_r$, 
it is also a generalized allowable sequence of permutations.
To a generalized allowable sequence of permutations there corresponds
a generalized wiring diagram.
This generalized wiring diagram is isomorphic
the original quasiline arrangement by construction.
\hfill $\Box$

%
%
%

\section{Quasi-topological incidence structures}
\label{section:quasiincidence}

In this section the class of quasi-topological configurations 
based on quasiline arrangements is introduced in a way parallel 
to topological configurations that are based on pseudoline arrangements.
An incidence structure is \emph{quasi-topological} if its points are points in
the projective plane, and lines are pseudolines that form a  quasiline arrangement.
A quasi-topological incidence structure is \emph{monotone} if its 
underlying quasiline arrangement is monotone.
A quasi-topological incidence structure  $\cQ$ is a \emph{quasi-topological realization} of
a combinatorial incidence structure $\cC$ if the underlying
combinatorial incidence structure of $\cQ$ is isomorphic to $\cC$.

The notions of combinatorial and topological equivalence for quasi-topological
incidence structures are the same as for topological incidence structures. 
However, since the pseudolines are allowed to cross more than once,
we may extend the notion of mutation equivalence. 
Also the local transformations where one psudoline moves across another
pseudoline such that they form a digon and no other pseudolines are crossed,
and their inverse transformations will be considered as admisible mutations.

Figure \ref{fig:fano} (a) shows a monotone quasi-topological realization
of the $(7_3)$ combinatorial configuration that cannot be realized as
a topological configuration; it  is polygonal with two bends.
However, the quasi-topological realization of the $(7_3)$ configuration 
from Figure \ref{fig:fano} (b) with a 7-fold rotational symmetry 
is not a monotone quasi-topological configuration, since
in every monotone orientation of the pseudolines for some of the pseudolines
the crossings will be in the opposite order. 

\begin{figure}[htb]
\centering
\begin{minipage}[c]{0.45\textwidth}
  \centering
  \includegraphics[width=0.7\textwidth]{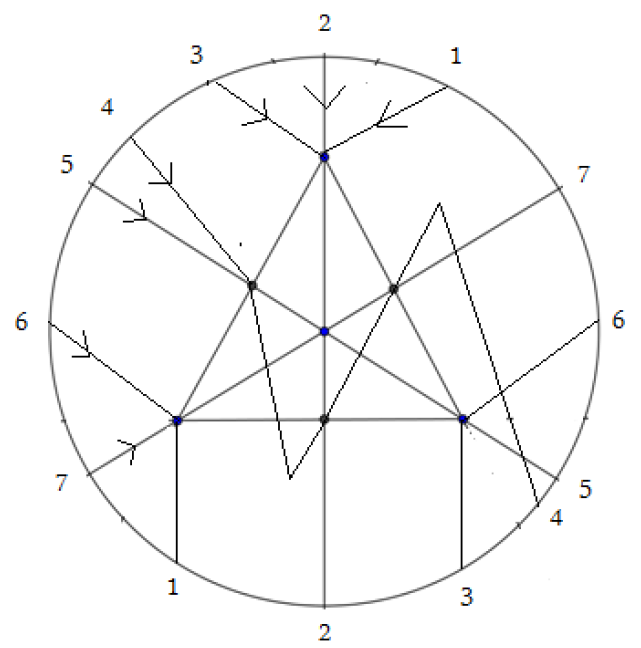}
\end{minipage}%
\begin{minipage}[c]{0.45\textwidth}
  \centering
  \includegraphics[width=0.7\textwidth]{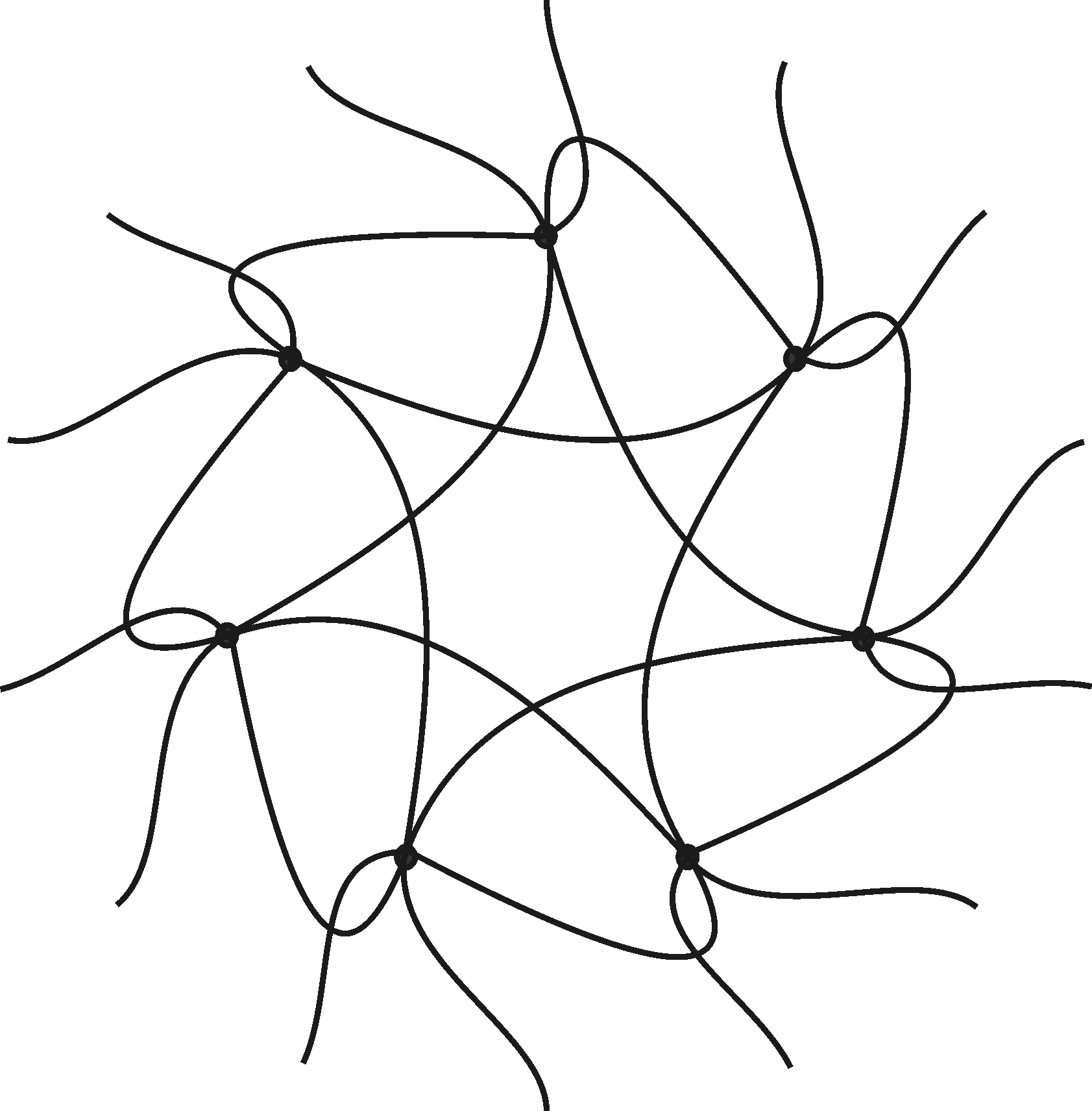}
\end{minipage}
\vspace*{2mm}
\begin{minipage}[t]{0.45\textwidth}
  \centering (a) 
\end{minipage}%
\begin{minipage}[t]{0.45\textwidth}
  \centering (b) 
\end{minipage}
\caption{Two different  quasi-topological realizations of the $(7_3)$ configuration.}
\label{fig:fano}
\end{figure}

By Theorem \ref{cor:monotonewiring} 
all monotone quasi-topological incidence structures 
can be represented with wiring diagrams (where only a subset of the crossings 
of the wiring diagram are considered as the points of the incidence structure).
Obviously, a topological incidence structure is also monotone quasi-topological.
The following theorem therefore shows that  the class of monotone quasi-topological 
incidence structures is in a sense the least generalization of the class of
topological incidence structures.

\begin{theorem}                            \label{thm:structurewiring}
A quasi-topological incidence structure can be represented by
a generalized wiring diagram if and only if it is monotone.
\end{theorem}
\begin{proof}
Let $\cQ$ be a quasi-topological incidence structure.
Suppose it can be represented by a generalized wiring diagram, i.e., it is
topologically equivalent to the  monotone quasiline arrangement
that corresponds to the generalized wiring diagram. Then it must
be monotone itself.

Conversely, let $\cQ$ be monotone. Then the underlying quasiline
arrangement $\cA$ is monotone and it can be represented by a generalized
wiring diagram by Theorem \ref{cor:monotonewiring}.
\end{proof}

We now consider the number of crossings in a quasiline arrangement.
Given a quasiline arrangement ${\cA}$, let $p$ be a crossing of ${\cA}$.
We define the \emph{local crossing number} of  $p$ to be ${ k \choose 2}$
if $k$ lines cross at $p$. Observe that the crossing number of a regular crossing is 1.
The \emph{crossing number} of the quasiline arrangement ${\cA}$ is 
the sum over all crossing numbers of its vertices.
Given a quasi-topological incidence structure, every pair of pseudolines
has at least one point in common. If such a point is not a point of the
incidence structure, we call it an \emph{unwanted crossing}.

\begin{proposition}
Given a topological $(n_k)$ configuration, where all the unwanted crossings
are regular, the number of unwanted crossings is
${ n \choose 2} - n {k \choose 2}$.
\end{proposition}


\section{Combinatorial incidence structures as quasiline arrangements}
\label{section:combquasiline}

Not every lineal combinatorial incidence structure can be realized 
as a topological incidence structure. Such examples are the well-known 
configurations $(7_3)$, the Fano plane, and $(8_3)$, the
M\"obius-Kantor configuration; see for example \cite[Theorem 2.1.3]{Grunbaum}. 
In this section we show that every combinatorial incidence structure can be realized
as a monotone quasi-topological incidence structure. In view of 
Theorem  \ref{thm:structurewiring},  monotone quasi-topological incidence structures
are in a sense the least generalization of topological incidence structures
with the property that any combinatorial incidence structure has a realization
within this class.

\begin{theorem}                                 \label{thm:combinatorialQuasiline}
Every combinatorial incidence structure can be realized as
a monotone quasi-topological incidence structure in the projective plane.
Moreover, the  order of pseudolines that come to each point of the quasi-topological 
incidence structure may be prescribed.
\end{theorem}

\begin{Proof}
Let $\cC$ be a combinatorial incidence structure and let $v$ be 
the number of points and $n$ be the number of lines of $\cC$.
Without loss of generality we may number the lines of $\cC$ 
by numbers $1, \dots,n$. To each point of $\cC$ we assign a 
substring $M_i$ of $\{1, \dots,n\}$, which corresponds to the incident lines 
of that point in prescribed order; if the  order of lines around points
is not prescribed, it may be arbitrary.
We find a quasi-topological realization of $\cC$ in the following way:
\begin{itemize}
\item Let $\pi_1, \dots, \pi_v$ be permutations of $\{1, \dots,n\}$
      with the property that elements from $M_i$ appear consecutively
      in $\pi_i$. Let $\pi_i^\prime $ be obtained from $\pi_i$
      by reversing the order of elements from $M_i$.
      Let $\pi_0^\prime$ be the identity permutation and
      $\pi_{v+1}$ the reverse permutation on $\{1, \dots, n\}$.
\item Form the sequence $\pi_0^\prime,\pi_1,\pi_1^\prime, \dots, \pi_v,\pi_v^\prime,\pi_{v+1}$;
      if $\pi_i^\prime = \pi_{i+1}$, take just one of them.
      If $\pi_i^\prime \ne \pi_{i+1}$, insert a sequence of permutations
      between them to obtain a generalized allowable sequence, for $i=0,\dots v$.
      This can always be done, in a bubble sort like manner by  interchanging
      two adjacent numbers in every step.
\item A generalized allowable sequence corresponds to a generalized wiring diagram
      which in turn corresponds to a monotone quasi-topological incidence structure.
      The points of the incidence structure correspond to the crossings
      $M_i$, $i=1,\dots, v$.
\end{itemize}

\end{Proof}

The proof of Theorem \ref{thm:combinatorialQuasiline} in fact provides
us with an algorithm to construct an actual quasi-topological representation
of a given combinatorial incidence structure. However, the number of
unwanted crossings in such a quasi-topological incidence structure can be high.
The following problem is therefore natural to consider.

\begin{problem}
For a given combinatorial incidence structure determine the minimal
number of crossings to realize it as a quasi-topological incidence structure
in the projective plane.   
\end{problem}

By Proposition  \ref{prop:nobends} any topological incidence structure can
be realized as a polygonal pseudoline arrangement with no bends.
With polygonal quasi-topological incidence structures we can not always avoid bends. 
For example, if  a quasi-topological incidence structure consists of three points 
and two polygonal pseudolines joining them, there have to be  bends,
since such a quasi-topological incidence structure contains two digons.
We will show that digons are in fact the only reason for the need of bends for
monotone quasi-topological incidence structures.
First we prove two technical lemmas.

\begin{lemma}                          \label{lemma:alllines}
Let $({\cA}, x^+)$ be a proper marked quasiline arrangement that contains no digons.
Then no crossing of the arrangement  is incident to all of its pseudolines. 
\end{lemma}  
\begin{Proof} 
Suppose there is a crossing $v$  in $\cA$ that is incident to
all the pseudolines. Then there has to be at least one more crossing,
otherwise  each pair of consecutive pseudolines through $v$ forms
a digon, crossing infinity. To form another crossing, at least two pseudolines 
that are consecutive in the cyclic order around $v$ must cross.
At least one crossing $w$ will be such that there is no other crossing 
between $v$ and $w$ on two consecutive lines through $v$. They form
a digon with vertices $v$ and $w$, a contradiction.
\end{Proof}

\begin{lemma}                          \label{lemma:digons}
Let $({\cA}, x^+)$ be a proper marked quasiline arrangement that contains no digons.
Let  $G$ be the undirected plane graph underlying its associated directed graph
$D=D({\cA}, x^+)$. Then the following hold.
\begin{itemize}
\item[{\rm (i)}]  Graph $G$ contains no cycles of length two, i.e., $G$ is simple.
\item[{\rm (ii)}] Graph $G$ is $2$-connected.
\end{itemize}
\end{lemma}  
\begin{Proof} (i) We will show that if $G$ contains a cycle of length two, 
then it contains a face of length two which corresponds to a digon in $\cA$.
Suppose $G$ contains a cycle $C$ of length two with vertices $v_1$ and $v_2$.  
If $C$ is the boundary of a face, we are done. Otherwise at least one 
pseudoline connects vertices $v_1$ and $v_2$ in the interior of $C$. 
If there are no vertices of the arrangement in the interior of $C$, 
we have at least two faces of length two in $G$.
Otherwise there is a vertex inside $C$ where at least two pseudolines cross. 
To form a crossing, at least two pseudolines that are consecutive 
in the cyclic order around $v_1$ must cross.
At least one crossing $w$ will be such that there is no other crossings 
between $v_1$ and $w$ on two consecutive lines through $v$. They form
a face of length two of $G$ inside $C$  with vertices $v_1$ and $w$.
\medskip

\noindent
(ii) Suppose $G$ is not 2-connected and let $v$ be a cutvertex in $G$.
By Lemma \ref{lemma:alllines} not all the pseudolines pass through $v$, 
so there exists also other pseudolines, say $\ell_1,\dots,\ell_k$, 
not incident to $v$. 
Some lines of the arrangement $\cA$ must cross before $v$
and some must cross after $v$, otherwise $v$ is not a cutvertex in $G$.
If only some of the pseudolines through $v$ cross before $v$, 
we have a digon by a similar reasoning as in (i). 
Therefore  a pseudoline through $v$ either has no crossing before $v$ 
or it is crossed by some of the lines $\ell_1,\dots,\ell_k$.
The same holds for the crossings after $v$. 
Since also the lines $\ell_1,\dots,\ell_k$ cross each other at
least once, the graph $G \backslash v$ is connected. A contradiction.
\end{Proof}

\begin{Remark}
Note that the requirement that the order of the intersection points 
of any two lines is the same for both lines in Lemma \ref{lemma:alllines} 
and Lemma \ref{lemma:digons}  is necessary, since otherwise 
it may happen that no two pseudolines that are consecutive
around a vertex $v$ form a digon; see Figure \ref{fig:nodigons}.
\end{Remark}

\begin{figure}[htb]
\centering
  \includegraphics[width=0.4\textwidth]{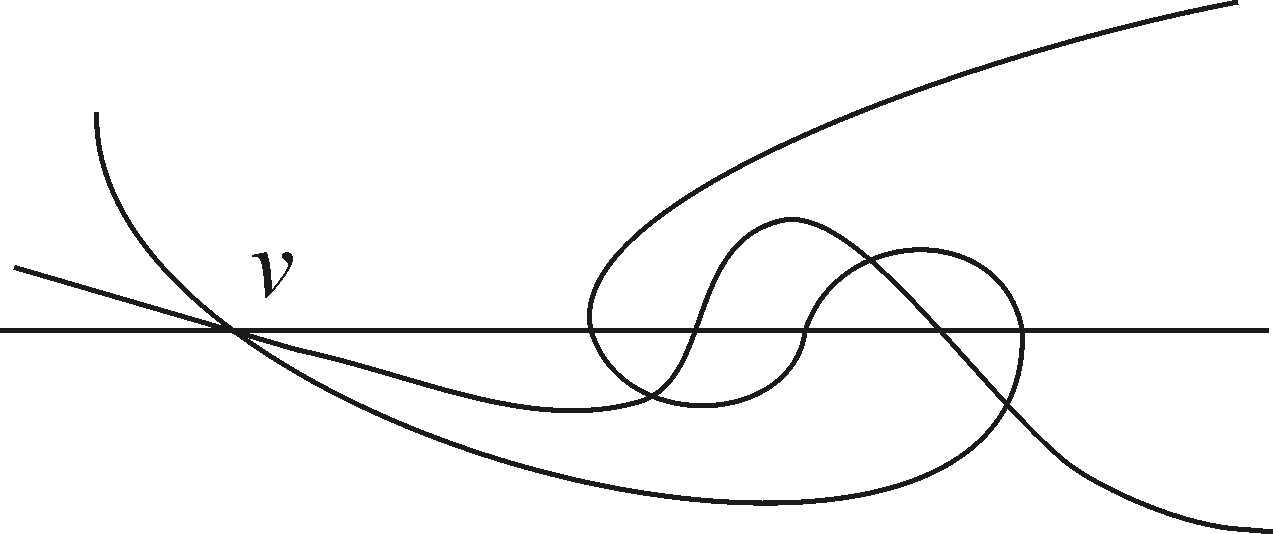}
\caption{No two consecutive pseudolines through vertex $v$ form a digon.}
\label{fig:nodigons}
\end{figure}

\begin{theorem}
Let $\cQ$ be a monotone quasi-topological incidence structure.
Then $\cQ$ is  topologically equivalent to a monotone polygonal quasi-topological 
incidence structure with no bends if and only if the underlying 
quasiline arrangement contains no digons.
\end{theorem}
\begin{Proof}
Let $\cA$ be the underlying quasiline arrangement of $\cQ$.
We may assume that no crossings of $\cA$ are points at infinity.

If $\cA$ contains a digon, there has to be a bend. 

Conversely, suppose there are no digons in $\cA$. 
Define a graph $G=G(\cA)$ as follows. The vertices of $G$ are the vertices 
of $\cA$ and there is an edge for every pair 
of vertices $u$ and $v$ that are consecutive on an arc that has 
an empty intersection with the line at infinity of some pseudoline
(we assume that every pseudoline of  $\cA$ intersects the line
at infinity exactly once).
Oberve that this graph is embedded in the plane.

Since $\cA$ is a monotone quasiline arrangement,
graph $G$ can be viewed as the underlying graph of the digraph
$D({\cA}, x^+)$ for some choice of $x^+$ such that $({\cA}, x^+)$
is a proper marked arrangement.
Graph $G$ is simple and 2-connected  by Lemma  \ref{lemma:digons}. 
Therefore we can  draw  $G$ in the plane in such a way that the vertices on
the boundary of the outer face are the vertices of a convex polygon $P$,
all the edges of $G$ are straight lines and the cyclic orders of edges
around each vertex is preserved by  \cite{Chiba}. 
To obtain a polygonal quasiline arrangement topologically equivalent to $\cA$
we draw the part corresponding to $G$ with straight lines as above.
To extend it to the  whole arrangement  we have to add the arcs crossing
the line at infinity that were omitted. These arcs connect
pairs of boundary vertices of $P$; to each pseudoline of the arrangement 
there corresponds a pair of boundary vertices. It cannot happen that some 
line only has one common vertex with $P$, since in that case all the 
pseudolines would share a common vertex. We connect these pairs of boundary 
vertices with straight lines and omit the parts of them in the interior of $P$, 
to obtain the missing arcs. In that way we assure that there are no bends
at infinity. Note that any pair of  pseudolines can have at most one  
common vertex of $P$, since there are no digons in $\cA$. 
Therefore all the lines are distinct.
The order in which the pseudolines enter $P$ is the same as the order
in which they leave $P$ and this order is therefore reflected in the
order of the straight lines that represent them. That also means that
no two of these straight lines are parallel and they intersect only
in the interior (these parts are omitted) or in the vertices of $P$ 
since $P$ is convex. The polygonal quasiline arrangement that we obtained 
therefore has no bends. Since the orders of lines around each crossing is preserved,
it is homeomorphic to the original quasiline arrangement by  \cite[Theorem 3.3.1]{Mohar}.
\end{Proof}

A quasi-topological configuration obtained from the proof of 
Theorem \ref{thm:combinatorialQuasiline} can have many digons.
Is it possible to avoid digons? Is it any easier if we only consider
lineal incidence structures?

\begin{problem}
Is every lineal combinatorial incidence structure  realizable as a polygonal
quasi-topological incidence structure in the projective plane 
with no bends? 
\end{problem}

\begin{problem}
Is every  combinatorial incidence structure  realizable as a polygonal
quasi-topological incidence structure in the projective plane 
with no bends? 
\end{problem}

If all the unwanted crossings are straight, a polygonal quasi-topological 
incidence structure is determined by the coordinates of the
original vertices (and the directions at which the pseudolines approach infinity). 
Therefore also the following two problems are of interest.

\begin{problem}
Is every lineal combinatorial incidence structure realizable as a polygonal
quasi-topological incidence structure in the projective plane 
with no bends and all the unwanted crossings straight?
\end{problem}

\begin{problem}
Is every topological incidence structure topologically equivalent to 
a polygonal topological incidence structure in the projective plane
with no bends and all the unwanted crossings straight?
\end{problem}


\section{Quasi-topological incidence structures as systems of curves on surfaces}

A \emph{curve arrangement} is a collection of simple closed curves on a 
given surface such that each pair of curves $(c_i, c_j )$, $i\ne j$ , 
has at most one point in common at which they cross transversely. 
In addition the arrangement should be cellular,
i.e., the complement of the curves is a union of open discs;
see J. Bokowski and T. Pisanski \cite{BokowskiPisanski}.

In this section we show that to any quasi-topological incidence
structure we can associate a map $M$ on a closed surface $S$ that can
be used to distinguish between mutation classes of quasi-topological configurations. 
Such a map defines an arrangement of curves on the surface $S$; 
i.e., every quasi-topological incidence structure can be viewed as a 
curve arrangement on some surface. 
This is a generalization of  the work of  J. Bokowski and R. Strausz 
\cite{BokowskiStrausz}, where only topological  configurations were considered.
For the background on graphs and maps we refer the reader to 
B. Mohar \cite{Mohar} or J. L. Gross and T. W. Tucker \cite{GT}.

Let $\cQ$ be a quasi-topological incidence structure. 
We define a graph $G=(V,E)$ corresponding to $\cQ$ in the following way.
The vertices of $V$ are the points of $\cQ$. Two vertices are connected
by an edge if the corresponding points of $\cQ$ are consecutive on some
pseudoline (the intersection points of pseudolines that are not
points of the incidence structure are ignored). Note that $G$ is Eulerian
since every pseudoline contributes  two edges through a vertex.
For each vertex $v$ we choose an orientation, which defines a
cyclic order $\pi_v$ of edges around $v$. The cyclic orders of all the vertices
form a \emph{rotation system} $\pi=\{\pi_v; v \in V\}$. Now define a signature mapping
$\lambda: E \to \{1,-1\}$ in the following way. For each edge $e=uv$ we 
check if the orientations at $u$ and $v$ agree if we move from $u$ to $v$
along $e$. If they agree we set $\lambda(e)=1$, otherwise $\lambda(e)=-1$. 
The pair $\Pi=(\pi,\lambda)$ is an \emph{embedding scheme} of $G$.
This defines a map, we denote it by $M=M(\cQ)$, on some surface $S$. 
This map is uniquely determined, up to homeomorphism; 
see B. Mohar \cite[Theorem 3.3.1]{Mohar}.
Note the surface $S$ is non-orientable. This can be seen as follows. 
Take a cycle $c$ of the map $M$ that corresponds to a pseudoline.
Since this is an orientation-reversing curve, starting at a vertex $v$ and
traveling along $c$, after returning to $v$  the orientation at $v$ is reversed.
That means that there must be an odd number of edges on $c$ with negative signature.
Consequently the embedding of $M$ is non-orientable by \cite[Lemma 4.1.4]{Mohar}.

A \emph{straight-ahead walk} or a SAW in an Eulerian map is a walk that always passes 
from an edge to the opposite edge in the rotation at the same vertex;
see Pisanski et al \cite{PisanskiTZ}.
Since the pseudolines of $\cQ$ cross transversally at each crossing of $\cQ$,
every pseudoline corresponds to a straight-ahead walk in the map $M(\cQ)$.
We have shown the following.

\begin{theorem}
For any quasi-topological incidence structure  $\cQ $ 
with the set of points $\cP$ and the set of lines $\cL$  
there exists an Eulerian map $M=M(Q)$ on a closed surface, 
with skeleton $G = (V,E)$ such that $\cP = V$,  and each $SAW$ 
is a simple closed curve that corresponds to a line from $\cL$.
\end{theorem}

The maps corresponding to quasi-topological incidence structures 
can be used to distinguish between quasi-topological
incidence structures that are not mutation equivalent.

\begin{theorem}
If two quasi-topological incidence structures  ${\cQ}_1$ and ${\cQ}_2$
are mutation equivalent, then $M({\cQ}_1)=M({\cQ}_2)$.
\end{theorem}
\begin{Proof}
Admissible mutations do not change the cyclic order of the pseudolines around 
any crossing that is  a point of the incidence structure.
On the other hand,  the cyclic orders of the pseudolines around every vertex 
defines the rotation systems for maps $M({\cQ}_1)$ and $M({\cQ}_2)$.
Moreover, if we choose local rotations consistently, also the embedding
schemes are equal.
Two maps with equal embedding schemes are considered to be the same.
\end{Proof}

\begin{example}
Consider quasi-topological configurations from Figure \ref{fig:fano}.
The corresponding maps have 7 vertices and 21 edges. The map corresponding to
the quasi-topological  configuration on the left has six faces of length three and two faces
of length 5, so it has Euler characteristic -5 and thus it has nonorientable
genus equal to 7. 
The map corresponding to quasi-topological configuration on the right 
has seven faces  of length five and one face of lenght seven, so it has 
Euler characteristic -6  and thus it has nonorientable genus equal to 8.
Therefore the two quasi-topological configurations are not mutation
equivalent.
\end{example}

For a given quasi-topological incidence structure  $\cQ $
the map $M(Q)$ can be viewed as a curve arrangement on some
nonorientable surface. Since every combinatorial incidence
structure can be realized as a quasi-topological incidence structure,
we can define  the (\emph{non})\emph{orientable genus of the incidence structure} 
$\cC$ as the smallest $g = g(\cC)$ for which there exists
a (non)orientable surface of genus $g$ on which $\cC$ can be represented with
a curve arrangement. 

\begin{problem}
For a given combinatorial incidence structure determine its
$($non$)$ orientable genus.
\end{problem}


\section*{Acknowledgements}


The authors would like to thank Sa\v so Strle for the fruitfull discussions
regarding the topological aspects of the paper.
This work was supported in part by
`Agencija za raziskovalno dejavnost Republike Slovenije', 
Grants  P1--0294, L1--4292 
and by European Science Foundation, Eurocores Eurogiga - GReGAS, N1--0011.

\noindent
{\sc J\"urgen Bokowski}\\
Department of Mathematics, Technische Universit\"at Darmstadt,\\
Schlossgartenstrasse 7, D--64289 Darmstadt, Germany\\
\texttt{juergen.bokowski@gmail.com}\\

\noindent
{\sc Jurij Kovi\v{c}},   \\
IMFM, University of Ljubljana, Jadranska 19, 1000 Ljubljana, Slovenia,\\
and\\
Andrej Maru\v si\v c Institute, University of Primorska, \\
Muzejski trg 2, 6000 Koper, Slovenia\\
\texttt{jurij.kovic@siol.net}\\

\noindent
{\sc Toma\v z Pisanski,} \\ 
Faculty for Mathematics and Physics, University of Ljubljana, \\
Jadranska 19, 1000 Ljubljana, Slovenia,\\ 
IMFM, University of Ljubljana, Jadranska 19, 1000 Ljubljana, Slovenia,\\
and\\
Andrej Maru\v si\v c Institute, University of Primorska, \\
Muzejski trg 2, 6000 Koper, Slovenia\\
\texttt{tomaz.pisanski@fmf.uni-lj.si}\\

\noindent
{\sc Arjana \v Zitnik}\\ 
Faculty for Mathematics and Physics, University of Ljubljana,\\
Jadranska 19, 1000 Ljubljana, Slovenia,\\ 
and\\
IMFM, University of Ljubljana, Jadranska 19, 1000 Ljubljana, Slovenia\\
\texttt{arjana.zitnik@fmf.uni-lj.si}\\ 


\begin{thebibliography}{99}


\bibitem{Chiba}
N. Chiba, K. Onoguchi, and T. Nishizeki,
Drawing plane graphs nicely,
{\em Acta Informatica} {\bf 22} (1985)  187--201.



\bibitem{BokowskiGrunbaum}
J. Bokowski, B. Gr\"unbaum, and L. Schewe,
Topological configurations $n_4$ exist for all $n \ge 17$,
European J.  Combin. {\bf 30} (2009) 1778--1785.

\bibitem{BokowskiPilaud}
J. Bokowski and V. Pilaud,
Enumerating topological $(n_k)$-configurations,
{\em Comput. Geom.} {\bf 47} (2014 ) 175--186.

\bibitem{BokowskiPilaud1}
J. Bokowski and V. Pilaud,
On topological and geometric $(19_4)$-configurations,
to appear in {\em European J. Combin.}, 
available at \texttt{ arXiv:1309.3201}.


\bibitem{BokowskiPilaud2}
J. Bokowski and V. Pilaud,
Quasi-configurations: building blocks for point-line configurations, 
submitted,
available at \texttt{arXiv:1403.7939}. 

\bibitem{BokowskiPisanski}
J. Bokowski and T. Pisanski,
Oriented matroids and complete-graph embeddings on surfaces,
{\em J. Combin. Theory Ser. A}  {\bf 114} (2007) 1--19.

\bibitem{BokowskiSchewe}
J. Bokowski and L. Schewe, 
On the finite set of missing geometric configurations $(n_4)$,
{\em Comput. Geom.} {\bf  46} (2013) 532--540.

\bibitem{BokowskiStrausz}
J. Bokowski and R. Strausz,
A manifold associated to a topological $n_k$-configuration,
{\em Ars Math. Contemp.} {\bf  7} (2014) 479--485.


\bibitem{BokowskiBook}
J. Bokowski and B. Sturmfels, 
{\em Computational synthetic geometry},
Lecture Notes in Mathematics 1355, Springer-Verlag, Berlin, 1989. 


\bibitem{Felsner}
S. Felsner, H. Weil,
Sweeps, arrangements and signotopes,
{\em Discrete Appl. Math.} {\bf 109} (2001) 67--94.

\bibitem{handbook}
J.~E. Goodman, 
Pseudoline arrangements,
in:
{\em Handbook of Discrete and Computational Geometry} 
(J.~E. Goodman, J. O'Rourke, eds.),
Chapman \& Hall/CRC, Boca Raton, FL, 2004, pp. 97--128.

\bibitem{GoodmanPollack}
J.~E. Goodman and R. Pollack,
Semispaces of configurations, cell complexes and arrangements,
{\em J. Combin. Theory Ser. A}  {\bf 37} (1984) 257--293.

\bibitem{Gropp1}
H. Gropp, Configurations and graphs,
{\em Discrete Math.} {\bf 111} (1993) 269--276. 

\bibitem{Gropp2}
H. Gropp, Configurations and graphs II,
{\em Discrete Math.} {\bf 164} (1997) 155--163.  

\bibitem{GT}
J.~L.~Gross and T.~W.~Tucker,
{\em Topological Graph Theory,} J. Wiley \& Sons, 1987.

\bibitem{Grunbaum72}
B. Gr\"unbaum,
{\em Arrangements and spreads,}
Volume 10 of {\em CBMS Regional Conf. Ser. in Math.},
American Mathematical Society, Providence, 1972. 

\bibitem{Grunbaum}
B. Gr\"unbaum,
{\em Configurations of Points and Lines,}
Graduate Studies in Mathematics, vol. 103, American Mathematical Society,
Providence, RI, 2009.

\bibitem{Grunbaum1}
B. Gr\"unbaum,
Connected $(n_4)$ configurations exist for almost all $n$ -- second update,
{\em Geombinatorics} {\bf 16} (2006) 254--261.

\bibitem{Levi}
F. Levi,
{\em Geometrische Konfigurationen,}
Hirzel, Leipzig, 1929.

\bibitem{Mohar}
B. Mohar and C. Thomassen,
{\em Graphs on surfaces,}
Johns Hopkins University Press, 2001.


\bibitem{PisanskiServatius}
T. Pisanski and B. Servatius,
{\em Configurations from a graphical viewpoint, }
Birkh\"auser advanced texts,  Birkh\"auser, New York, 2013.

\bibitem{PisanskiTZ}
T. Pisanski, T.~W. Tucker and A. \v Zitnik,
Straight-ahead walks in Eulerian graphs, 
{\em Discrete Math.} {\bf 281} (2004) 237--246.

\bibitem{Reye}
T. Reye,
Das Problem der Configurationen, 
{\em Acta. Math.} {\bf 1} (1882) 93--96.


\bibitem{Schroter2}
H. Schr\"oter,
\"Uber die Bildungsweise und geometrische Konfiguration $(10_3)$,
{\em Nachr. Ges. Wiss. G\"ottingen} 1889, 239--253.









\end{thebibliography}
\end{document}